\documentclass{amsart}

\usepackage[inactive]{srcltx} % SRC Specials for DVI Searching
\usepackage{cite}              % to include figures
\usepackage{amsmath, amsthm, amscd, amsfonts, amssymb}
 \newtheorem{thm}{Theorem}[section]

 \newtheorem{prop}[thm]{Proposition}
 \theoremstyle{definition}
 
 \newtheorem{rem}[thm]{Remark}
 \newtheorem{exm}[thm]{Example}
 \theoremstyle{remark}
 
\newcommand{\A}{\mathcal{A}}      
\newcommand{\U}{\mathcal{U}}

\DeclareMathOperator{\id}{id}

\begin{document}

\nocite{*}

\title[On automatic continuity of the derivations on module extension of Banach algebras]{On automatic continuity of the derivations on module extension of Banach algebras}
\author{ H. Farhadi }
\thanks{{\scriptsize
\hskip -0.4 true cm \emph{MSC(2010)}:  16E40; 46H40; 46H25.
}}

\address{Department of
Mathematics, University of Kurdistan, P. O. Box 416, Sanandaj,
Iran.}
\email{h.farhadi@sci.uok.ac.ir}
\maketitle
\begin{abstract}
In this article we shall focus on the derivations on module extension of Banach algebras and determine the general structure of them. Then we obtain some results concerning the automatic continuity of these mappings.
\end{abstract}
\section{Introduction}
Let $\A$ be a Banach algebra (over the complex field $\mathbb{C}$), and $\U$ be a Banach $\A$-bimodule. A linear mapping $D:\A\rightarrow\U$ is called a \textit{derivation} if $D (ab)=a D(b)+D (a)b$ holds for all $a,b\in\A$. If $\U=\A$, then $D$ is said to be a derivation on Banach algebra $\A$. For any $u\in\U$, the mapping $\id_u:\A\rightarrow\U$ given by $\id_u (a)=au-ua$ is a continuous derivation called \textit{inner derivation} induced by $u$. Derivations and their various properties are significant subjects in the study of Banach algebras. One of the most important problems related to these mappings is this question; Under what conditions is a derivation $D:\A\rightarrow\U$ continuous? This question lies in the theory of automatic continuity which is an important subject in mathematical analysis and  has attracted the attention of researchers during the last few decades. In this theory we are looking for conditions under which one can conclude that a linear mapping between two Banach algebras (or two Banach spaces, in general) is necessarily continuous. For more information, the reader may refer to \cite{da} which is a detailed source in this context. Here we mention the most important established results in this regard. A celebrated theorem due to Johnson and Sinclair \cite{john1} states that every derivation on a semisimple Banach algebra is continuous. Later on, Ringrose in \cite{rin} showed that every derivation from a $C^{*}$-algebra $\A$ into Banach $\A$-bimodule $\U$ is continuous. In \cite{chr}, Christensen proved that every derivation from nest algebra on Hilbert space $\mathcal{H}$ into $\mathbb{B}(\mathcal{H})$ is continuous. Additionally, some results on automatic continuity of the derivations on prime Banach algebras have been established by Villena in \cite{vi1} and \cite{vi2}. See also \cite{gh2, gh8} where the continuity of the derivations are studied on certain rpoducts of Banach algebras. Some results on  the derivations and Jordan derivations on  trivial extension and triangular Banach algebras have been established by  in a number of papers; see \cite{gh1, gh3, gh4, gh6, gh7, gh8, gh9, gh10, gh11, gh12, gh13}
 
 Let $\A$ be a Banach algebra and $\U$ be a Banach  $\A$-bimodule. The \textit{module extension or trivial extension Banach algebra} associated to $\A$ and $\U$, denoted by $T(\A,\U)$, is the $\ell^1$-direct sum $\A\oplus\U$ equipped with the algebra multiplication given by 
\[(a,u)(b,v)=(ab,av+ub)\quad\quad (a,b\in\A, u,v\in\U).\]
and the norm
\[||(a,u)||=||a||+||u|| \quad \quad (a\in\A,u\in\U).\] 
Many authors have studied this class of Banach algebras from different points of view and investigated various properties of these spaces; see for example \cite{med, zh}.
\par 
At the end of this section we introduce some notations and expressions used in the paper.
Let $\A$ be a Banach algebra and $\U, \mathcal{V}$ be  Banach $\A$-bimodules. The annihilator of $\U$ over $\A$ is defined as
\[ann_{\A}\U=\{a\in\A\, \mid \, a\U=\U a=(0)\}.\]
Recall that a linear mapping $\phi :\U\rightarrow \mathcal{V}$ is said to be a \textit{left $\A$-module homomorphism} if $\phi (au)=a\phi (u)$ whenever $a\in \A$ and $u\in\U$ and it is a \textit{right $\A$-module homomorphism} if $\phi (ua)=\phi (u) a\quad (a\in \A,u\in\U)$. $\phi$ is called \textit{$\A$-module homomorphism}, if $ \phi$ is both left and right $\A$-module homomorphism. The set of all continuous $\A$-module homomorphisms from $\U$ into $\mathcal{V}$ is denoted by $Hom_{\A}(\U,\mathcal{V})$.
%%****************************************************************************************
\section{Continuity of the derivations }
Throughout this section $\A\, ,\U$ are Banach algebras and $T(\A,\U)$ denotes the corresponding module extension Banach algebra. We commence with the following proposition characterizing the derivations on $T(\A,\U)$.
\begin{prop}\label{asli}
Let $D:T(\A,\U)\rightarrow T(\A,\U)$ be a linear mapping such that 
 \[D((a,u))=(\delta_1(a)+\tau_1(u),\delta_2(a)+\tau_2(u))\quad\quad ((a,u)\in \A\times\U)\] in which $\delta_1:\A\rightarrow \A$, $\delta_2:\A\rightarrow \U$, $\tau_1:\U\rightarrow \A$ and $\tau_{2}:\U\rightarrow \U$ are linear mappings. Then the following are equivalent.
 \begin{enumerate}
 \item[(i)]
 $D$ is a derivation. 
 \item[(ii)]
 $\delta_1:\A\rightarrow \A$, $\delta_2:\A\rightarrow \U$ are derivations, $\tau_{2}:\U\rightarrow \U$ is a linear mapping satisfying 
 \[\tau_2 (au)=a\tau_2 (u)+\delta_2 (a)u\quad \text{and} \quad \tau_2 (ua)=\tau_2 (u)a+u\delta_2(a)\quad \quad (a\in\A,u\in\U).\] 
 $\tau_1:\U\rightarrow \A$ is an $\A$-bimodule homomorphism for which $u\tau_1(v)+\tau_1(u)v=0$ $ (u,v\in\U)$. 
\end{enumerate}
 As a consequence, any derivation $D:T(\A,\U)\rightarrow T(\A,\U)$ can be written as 
 $D=D_1+D_2$ where $D_1((a,u))=(\delta_1(a)+\tau_1(u),\tau_2(u))$ and $D_{2}((a,u))=(0,\delta_2(a))$ are derivations on $T(\A,\U)$.
 Moreover, $D$ is an inner derivation if and only if $\delta_1 ,\delta_2$ are inner derivations and $\tau_1 =0$.
 \end{prop}
 \begin{proof}
 The proof follows from a straightforward verification. 
 \end{proof}
\begin{rem} \label{r0} 
Note that one can show that the linear mapping $\delta:\A\rightarrow\U$ is a derivation if and only if  $D:T(\A,\U)\rightarrow T(\A,\U)$ given by $D((a,u))=(0,\delta (a))$ is a derivation. So if any derivation on $T(\A,\U)$ is continuous, then any derivation $\delta:\A\rightarrow\U$ is continuous as well.
 \end{rem} 
 In the next proposition we give some sufficient conditions for derivations on $\A$ be continuous. 
 \begin{prop}\label{tau}
Suppose that there are $\A$-bimodule homomorphisms $\phi:\A\rightarrow\U$ and $\psi:\U\rightarrow\A$ such that $\phi o\psi = I_{\U}$ ($I_\U$ is the identity mapping on $\U$). If every derivation on $T(\A,\U)$ is continuous, then every derivation on $\A$ is continuous.
\end{prop}
\begin{proof}
Let $\delta$ be a derivation on $\A$. Define the map $\tau:\U\rightarrow\U$ by $\tau =\phi o\delta o\psi$.
Then for every $a\in\A,x\in\U$,
\[\tau(ax)=a\tau(x)+\delta(a)x\quad \text{and} \quad \tau(xa)=\tau(x)a+x\delta(a).\]
So the map $D:T(\A,\U)\rightarrow T(\A,\U)$ defined by $D((a,x))=(\delta(a),\tau(x))$ is a derivation which is continuous by the hypothesis. Thus $\delta$ is continuous.
\end{proof}
\begin{rem}\label{r1}
If $\A$ is a Banach algebra and $\mathcal{I}$ is a closed ideal on it, then $\frac{\A}{\mathcal{I}}$ is a Banach $\A$-bimodule and so we can consider $T(\A, \frac{\A}{\mathcal{I}})$. Suppose that $\A$ possesses a bounded right (or left) approximate identity and every derivation on $\A$ and every derivation from $\A$ to $\frac{\A}{\mathcal{I}}$ is continuous. Let $D$ be a derivation on $T(\A, \frac{\A}{\mathcal{I}})$ which has a structure as in Proposition \ref{asli}. Then $\tau_1:\frac{\A}{\mathcal{I}}\rightarrow \A$ is an $\A$-bimodule homomorphism. Since $\A$ has a bounded right(left) approximate identity, then so does $\frac{\A}{\mathcal{I}}$. Hence $\tau_1$ is continuous implying that $D$ is continuous. So in this case any derivation on $T(\A, \frac{\A}{\mathcal{I}})$ is continuous.
\end{rem}
\begin{rem}\label{r2}
If $\mathcal{I}$ is a closed ideal in a Banach algebra $\A$ and $\delta:\A\rightarrow\A$ is a derivation such that $\delta(\mathcal{I})\subseteq \mathcal{I}$, then the mapping $\tau: \frac{\A}{\mathcal{I}}\rightarrow \frac{\A}{\mathcal{I}}$ defined by $\tau(a+\mathcal{I})=\delta(a)+\mathcal{I}$ is well-defined and linear and 
\[\tau (a(u+\mathcal{I}))=a\tau (u+\mathcal{I})+\delta (a)(u+\mathcal{I})\quad,\quad \tau((u+\mathcal{I})a)=\tau(u+\mathcal{I})a+(u+\mathcal{I})\delta(a).\]
 Therefore the mapping $D$ on  $T(\A, \frac{\A}{\mathcal{I}})$ defined by $D((a,u))=(\delta (a),\tau(u+\mathcal{I}))$ is a derivation. So if every derivation on  $T(\A, \frac{\A}{\mathcal{I}})$ is continuous, then every derivation $\delta:\A\rightarrow \A $ with $\delta(\mathcal{I})\subseteq \mathcal{I}$ is continuous.
\end{rem}
 In Remarks \ref{r1} and \ref{r2}, if we let $\mathcal{I}=(0)$, then we have the following proposition.
 \begin{prop}
  Let $\A$ be a Banach algebra.
\begin{enumerate}
\item[(i)] If $\A$ has a right (left) approximate identity and every derivation on $\A$ is continuous, then any derivation on $T(\A, \A)$ is continuous. 
\item[(ii)] If every derivation on $T(\A, \A)$ is continuous, then any derivation on $\A$ is continuous.
\end{enumerate} 
 \end{prop}
\begin{prop}\label{trs}
Let $\A$ be a semisimple Banach algebra which has a bounded approximate identity and $\U$ be a Banach $\A$-bimodule with $ann_{\A}\U=(0)$. Suppose that there exists a surjective left $\A$-module homomorphism $\phi:\A\rightarrow\U$ and every derivation from $\A$ into $\U$ is continuous, then every derivation on $T(\A,\U)$ is continuous.
\end{prop}
\begin{proof}
Let $D$ be a derivation on $T(\A,\U)$ which has a structure as in Proposition \ref{asli}. Since $\A$ is semisimple, every derivation from $\A$ into $\A$ is continuous. Now by Proposition \ref{asli} and Remark \ref{r1}, it follows that every derivation on $T(\A,\U)$ is continuous.
\end{proof}
Ringrose \cite{rin} proved that every derivation from a $C^{*}$-algebra into a Banach bimodule is continuous. According to this result we give the following example satisfying the conditions of the above proposition.
\begin{exm}
Let $\A$ be a $C^{*}$-algebra and $\U$ be a Banach $\A$-bimodule with $ann_{\A}\U=(0)$. Suppose that there exists a surjective left $\A$-module homomorphism $\phi:\A\rightarrow\U$. Hence $\A$ is a semisimple Banach algebra with a bounded approximate identity. So by \cite{rin} and Proposition \ref{trs}, any derivation on $T(\A,\U)$ is continuous. 
\end{exm}

An non-trivial element $p$ in an algebra $\A$ is called an \textit{idempotent} if $p^{2}=p$.
\begin{prop}
Let $\A$ be a prime Banach algebra with a non-trivial idempotent
$p$ (i.e. $p \neq 0$) such that $\A p$ is finite dimensional. Then every derivation
on $\A$ is continuous.
\end{prop}
\begin{proof}
Let $\U=\A p$. Then $\U$ is a closed left ideal in $\A$. By the following make $\U$ into a Banach
$\A$-bimodule:
\[ xa=0 \quad \quad (x\in \A p, a\in \A), \]
and the left multiplication is the usual multiplication of $\A$. So we can consider $T(\A,\U)$ in this case. Since $\A$ is prime, it follows that $ann_{\A}\U=(0)$. Let $\delta:\A \rightarrow \A$ be a derivation. Define the mapping $\tau:\U \rightarrow \U$ by $\tau(ap)=\delta(ap)p$ $(a\in \A)$. The mapping $\tau$ is well-defined and linear. Also
\[\tau (ax)=a\tau (x)+\delta (a)x\quad \text{and} \quad \tau (xa)=\tau (x)a+x\delta(a)\quad \quad (a\in\A,x\in\U).\] 
Since $\U$ is finite dimensional, it follows that $\tau$ is continuous. By Proposition \ref{asli}, the mapping $D:T(\A,\U)\rightarrow T(\A,\U)$ defined by $D((a,x))=(\delta(a),\tau(x))$ is a derivation. Now for $D$, the conditions of Proposition \ref{asli} are all satisfied and so $D$ is continuous. Therefore $\delta$ is continuous.
\end{proof}
\bibliographystyle{plain}

\end{document}